\documentclass[12pt]{amsart}
\usepackage{graphics,wrapfig, graphicx, mathrsfs,xcolor}
\usepackage{mathtools}
\usepackage{amssymb,amscd,graphicx,xcolor,subfig,tikz}
\usepackage{graphics}
\usepackage{indentfirst}
\usepackage{bm, enumerate,xcolor}
\usepackage[dvips]{epsfig}
\usepackage{latexsym}
\usepackage{float}
\usepackage[colorlinks]{hyperref}
\AtBeginDocument{
   \hypersetup{
    linkcolor=blue,
    citecolor=blue,
 }
}

\usepackage{amsmath}
\usepackage{amssymb}
\usepackage[applemac]{inputenc}
\usepackage[T1]{fontenc}
\newtheorem{lemma}{Lemma}[section]
\newtheorem{theorem}{Theorem}[section]
\newtheorem{corollary}{Corollary}[section]

\newtheorem{definition}{Definition}[section]
\newtheorem{remark}{Remark}[section]

\numberwithin{equation}{section} \numberwithin{theorem}{section}
\numberwithin{example}{section} \numberwithin{remark}{section}
\numberwithin{figure}{section} \numberwithin{algorithm}{section}
\mathtoolsset{showonlyrefs}

\title[Classification in a Lipschitz epigraphical cone]{Classification of global solutions to the singular equation $-\Delta u=f(X)\cdot u^{-\gamma}$ in a Lipschitz epigraphical cone}
\author{Yahong Guo}
\address{School of Mathematical Sciences, Shanghai Jiao Tong University, Shanghai 200240, China}\email{yhguo@sjtu.edu.cn}
\author{Congming Li}
\address{School of Mathematical Sciences and CMA-Shanghai, Shanghai Jiao Tong University, China}\email{congming.li@sjtu.edu.cn}
\author{Chilin Zhang$^{\ast}$}\thanks{*Corresponding author}
\address{School of Mathematical Sciences, Fudan University, Shanghai 200433, China}\email{zhangchilin@fudan.edu.cn}

\begin{document}
\begin{abstract}
    We construct and classify all global solutions to the singular equation $$-\Delta u=f(X)\cdot u^{-\gamma}$$ supported in a general Lipschitz epigraphical cone, where $f(X)$ is a locally Dini continuous function with $0<\lambda\leq f(X)\leq\Lambda$. The existence and non-existence of a global solution is solely determined by the relateion between the scaling exponent $\alpha=\frac{2}{1+\gamma}$ and the ``frequency" of the cone.
    
    Moreover, in order to classify all global solutions, we introduce several new methods. First, we use the local data to estimate the global growth rate, which in turn establishes the boundedness of the ``asymptotic slope" of the global solution. Second, by establishing a nonlinear variant of Kemper's boundary Harnack principle, we classify all global solutions through an ``oscillation reduction" argument on the ``asymptotic slope".
\end{abstract}
\maketitle

\section{Introduction}
\subsection{Background}

In this paper, we study classical global solutions of the singular elliptic equation
\begin{equation}\label{eq. main}
    \left\{\begin{aligned}
        &-\Delta u=f(X)u^{-\gamma}\mbox{ and }u>0&\mbox{in }&Cone_{\Sigma},\\
        &u=0&\mbox{on }&\partial Cone_{\Sigma}.
    \end{aligned}\right.
\end{equation}
Where $Cone_{\Sigma}\subseteq\mathbb{R}^{n}$ is a Lipschitz epigraphical cone and $f(X)$ is a locally Dini continuous function satisfying $0<\lambda\leq f(X)\leq\Lambda$. Here and below, by a classical global solution, we mean $u\in C^2(Cone_{\Sigma})\cap C^0(\overline{Cone_{\Sigma}})$, positive in $Cone_{\Sigma}$, and attaining the prescribed zero boundary value continuously. We aim to determine the precise range of parameters for which such solutions exist, and also to provide a complete classification of all global solutions in the solvable range without imposing any a priori growth condition at infinity. Global cone solutions arise naturally as limiting profiles in blow-up analysis near non-smooth boundary points, our complete classification provides useful tools for local boundary regularity and compactness arguments.

The singular Lane-Emden-Fowler equation
\begin{equation}\label{eq. SLEF}
    -\Delta u=f(X)\cdot u^{-\gamma}
\end{equation}
has been studied since the work of Fulks--Maybee \cite{FuMa}, Stuart \cite{Stu}, and Crandall-Rabinowitz-Tartar \cite{CrRaTa}. It also appears in the study of pseudoplastic fluids \cite{NaCa}. Subsequent work developed existence, uniqueness, regularity, and energy theories for singular semilinear equations, see \cite{BO10,dP92,Edelson1989,G86,GL93,KusanoSwanson1985,LM91,OP18}. For a recent broad overview of the subject and further references, see Oliva--Petitta \cite{OP24}.

A distinctive feature of \eqref{eq. SLEF} is that the Dirichlet condition itself creates the singularity: as $u$ approaches zero at the boundary, the source term $u^{-\gamma}$ diverges. Consequently, the boundary growth rate is the key obstruction in the regularity of the solution. In smooth domains, this behavior is by now well understood. In particular, the works of Lazer-McKenna \cite{LM91}, del Pino \cite{dP92}, and Gui-Lin \cite{GL93} lead to the sharp trichotomy:
\begin{equation}\label{eq. Gui-Lin estimate}
    u(X)\sim
    \begin{cases}
        d(X)^{\frac{2}{1+\gamma}}, & \gamma>1,\\[3pt]
        d(X)\bigl|\log d(X)\bigr|^{1/2}, & \gamma=1,\\[3pt]
        d(X), & 0<\gamma<1,
    \end{cases}
    \qquad d(X):=\operatorname{dist}(X,\partial D)
\end{equation}
for a model problem in a $C^{1,1}$ bounded domain $D$. Thus, when $\gamma>1$, the natural scaling exponent $\alpha=\frac{2}{1+\gamma}$ controls the leading boundary behavior, while the linear scale becomes dominant for $0 <\gamma< 1$ and a logarithmic correction appears at $\gamma=1$. Interestingly, $\gamma=1$ also acts as the solvability threshold for the following global problem in the half-space:
\begin{equation}\label{eq. half space problem}
    \left\{\begin{aligned}
        &-\Delta u=u^{-\gamma}&\mbox{in }&\mathbb{R}^{n}_{+},\\
        &u=0&\mbox{on }&\partial\mathbb{R}^{n}_{+}.
    \end{aligned}\right.
\end{equation}
Montoro-Muglia-Sciunzi \cite{MMS24b,MMS24a} classified all global solutions when $\gamma>1$, and showed the non-existence for $\gamma\leq1$. Specifically, all global solutions to \eqref{eq. half space problem} must be one-dimensional, satisfying an ODE in the $x_{n}$ direction. Related classification results have recently been obtained in the fractional setting \cite{GW25,GuoZhang2026}.

%Entire positive solutions in the whole space and other unbounded domains were studied from a different viewpoint in \cite{Edelson1989,KusanoSwanson1985}. %The present paper, by contrast, considers the equation on a cone with zero boundary data and provides a complete classification of all global solutions in this setting, {\color{red} thereby extending the results of Montoro, Muglia, and Sciunzi \cite{MMS24b,MMS24a} from the half-space to a general Lipschitz epigraphical cone.}

The situation changes when the domain is merely Lipschitz. In particular, $\gamma=1$ is no longer necessarily the border case of the growth rate trichotomy as in \eqref{eq. Gui-Lin estimate}. Instead, our recent work \cite{GLZ25} shows that the comparison between the scaling exponent $\alpha=\frac{2}{1+\gamma}$ and the ``frequency" $\phi$ of the tangent cone (see Definition~\ref{def. frequency}) completely characterizes the local boundary growth estimates. It is therefore natural to expect that exactly the same comparison also governs the existence and nonexistence of global solutions to \eqref{eq. main}. In this paper, we show that global solutions to \eqref{eq. main} exist if and only if $\alpha<\phi$. Moreover, within this range, all global solutions form an ordered family $\{\Psi_{K}\}_{K\geq0}$, where the parameter $K$ is an intrinsic asymptotic harmonic slope, i.e.
\begin{equation}\label{eq. K means max in B_R}
    K=\lim_{R\to\infty}\big(R^{-\phi}\cdot\max_{B_{R}\cap Cone_{\Sigma}}\Psi_{K}\big)\in[0,\infty),
\end{equation}
A detailed formulation of our main result will be given in Theorem~\ref{thm. non-existing}-\ref{thm. classification}. Our result includes the classification in \cite{MMS24b,MMS24a} as a special case, see Corollary~\ref{cor. great extension, a glimpse}.

The main difficulty in our classification lies in the lack of symmetry of $Cone_{\Sigma}$ (other than scaling invariance) and that $f(X)$ is not a constant. Then, the ODE analysis in \cite{MMS24b} is no longer valid, and the Kelvin transform method in \cite{MMS24a} is no longer applicable. Instead, our construction of global solutions is based on a compactness argument, where the growth rate estimate in our previous work \cite{GLZ25} plays a fundamental role. Moreover, by developing a nonlinear variant of Kemper's boundary Harnack principle, we establish the growth rate estimate at infinity using solely local information, and classify global solutions to \eqref{eq. main} by reducing the oscillation of the ``asymptotic slope".

\subsection{Main results}
Before stating the main results, we need to clarify several notations and settings.
We first define a Lipschitz epigraphical cone.
\begin{definition}\label{def. Lipschitz epigraphical cone}
    Let $\Sigma\subseteq\partial B_{1}$ be an open spherical domain. Let $Cone_{\Sigma}$ be the cone of $\Sigma$ defined as
    \begin{equation*}
        Cone_{\Sigma}=\{X\in\mathbb{R}^{n}\setminus\{0\}:\frac{X}{|X|}\in\Sigma\}.
    \end{equation*}
    We say that $Cone_{\Sigma}$ is a Lipschitz epigraphical cone (in the $\vec{e}_{n}$-direction), if there exists an entire Lipschitz function $g:\mathbb{R}^{n-1}\to\mathbb{R}$, such that
    \begin{equation*}
        Cone_{\Sigma}=\{x_{n}>g(x')\}.
    \end{equation*}
    We also let $\Gamma$ be the boundary of $Cone_{\Sigma}$, i.e.:
    \begin{equation*}
        \Gamma=\partial Cone_{\Sigma}=\{x_{n}=g(x')\}.
    \end{equation*}
\end{definition}
\begin{remark}
    One can easily verify that for a Lipschitz epigraphical cone, the function $g$ satisfies $g(t\cdot x')=t\cdot g(x')$ for all $t\geq0$.
\end{remark}
We next define the ``frequency" of the cone $Cone_{\Sigma}$.
\begin{definition}\label{def. frequency}
    Let $Cone_{\Sigma}$ be a Lipschitz epigraphical cone. We let $\lambda_{\Sigma}$ be the first eigenvalue of $\Sigma$ with respect to the Laplace-Beltrami operator on the sphere, namely:
    \begin{equation}\label{eq. first eigenvalue}
        \lambda_{\Sigma}=\inf_{F\in H^{1}_{0}(\Sigma)\setminus\{0\}}\frac{\int_{\Sigma}|\nabla F(\theta)|^{2}d\theta}{\int_{\Sigma}|F(\theta)|^{2}d\theta}
    \end{equation}
    We then let $\phi_{\Sigma}>0$ be the unique positive solution to
    \begin{equation*}
        \phi_{\Sigma}(\phi_{\Sigma}+n-2)=\lambda_{\Sigma}.
    \end{equation*}
    We call $\phi_{\Sigma}$ the ``frequency" of $Cone_{\Sigma}$.
\end{definition}
The word ``frequency" is borrowed from the theory developed by Almgren, and it is used to describe the growth rate of a homogeneous harmonic polynomial. In fact, $\phi_{\Sigma}$ defined above is exactly the degree of a positive homogeneous harmonic function in $Cone_{\Sigma}$.
\begin{definition}\label{def. H(X)}
    We define $H_{\Sigma}(X)$ as a positive homogeneous harmonic function in $Cone_{\Sigma}$ in the following way. Let $E(\theta)\in H^{1}_{0}(\Omega)\cap C^{\infty}(\Sigma)\cap C^{0}(\overline{\Sigma})$ be a positive eigenfunction associated with $\lambda_{\Sigma}$, which is normalized such that $\displaystyle\max_{\theta\in\Sigma}E(\theta)=1$. Then we define
    \begin{equation}\label{eq. homogeneous harmonic H}
        H_{\Sigma}(X)=H_{\Sigma}(r,\theta)=r^{\phi_{\Sigma}}E(\theta)
    \end{equation}
    written in the polar coordinate with $(r,\theta)=(|X|,\frac{X}{|X|})$. Clearly, $H_{\Sigma}(X)$ is a positive homogeneous harmonic function in $Cone_{\Sigma}$ with degree $\phi_{\Sigma}$.
\end{definition}

Finally, we introduce a few cylindrical domains, see also \cite[Definition 1.3]{GLZ25}.
\begin{definition}\label{def. cylindrical domains}
    Let $\Gamma$ be a Lipschitz graph, given by $\{x_{n}=g(x')\}$ where $g$ is a Lipschitz function satisfying $g(0)=0$. For every $X=(x',g(x'))\in\Gamma$, we define
    \begin{itemize}
        \item A grounded cylinder $\mathcal{GC}_{r}(X)$ as
        \begin{equation*}
            \mathcal{GC}_{r}(X):=\{Y=(y',y_{n}):\quad|x'-y'|\leq r,\quad0\leq y_{n}-g(y')\leq r\};
        \end{equation*}
        \item A suspended cylinder $\mathcal{SC}_{r,\delta}(X)$ as
        \begin{equation*}
            \mathcal{SC}_{r,\delta}(X):=\{Y=(y',y_{n}):\quad|x'-y'|\leq r,\quad\delta r\leq y_{n}-g(y')\leq r\}.
        \end{equation*}
    \end{itemize}
    The small $\delta\leq\frac{1}{10}$ will be chosen later. If there is no risk of ambiguity (especially in the value of $\delta$), we will adopt the following abbreviations:
    \begin{equation*}
        \mathcal{GC}_{r}=\mathcal{GC}_{r}(0),\quad\mathcal{SC}_{r}=\mathcal{SC}_{r,\delta}=\mathcal{SC}_{r,\delta}(0).
    \end{equation*}
\end{definition}

Now we present the main results in this paper. For the sake of convenience, throughout the paper we denote
\begin{equation*}
    \alpha=\frac{2}{1+\gamma},\quad\phi=\phi_{\Sigma},\quad H(X)=H_{\Sigma}(X).
\end{equation*}
Our first theorem gives the sharp nonexistence range.
%Our first result is the non-existing theorem when the power $\gamma$ is small.
\begin{theorem}[Nonexistence]\label{thm. non-existing}
    Assume that $\alpha\geq\phi$ and $f(X)\geq\lambda>0$, then \eqref{eq. main} admits no global solution. 
\end{theorem}
The second result is the construction of global solutions to \eqref{eq. main} when $\alpha<\phi$.
\begin{theorem}[Existence]\label{thm. construction}
    Assume that $\alpha<\phi$, and $f(X)$ is a locally Dini function in $Cone_{\Sigma}$ satisfying $0<\lambda\leq f(X)\leq\Lambda$.  Then there is a family of global solutions $\{\Psi_K\}_{K\geq0}$ of \eqref{eq. main} with the following properties.%\eqref{eq. main} has a family of global solutions $\Psi_{K}(X)$ for $K\geq0$ such that:
    \begin{itemize}
        \item[(1)] The ground solution $\Psi_{0}$ satisfies$$\Psi_{0}(X)\leq C|X|^{\alpha}\ \mbox{in}\ Cone_{\Sigma},$$ where $C$ depends on $(\Sigma,\gamma,\lambda,\Lambda)$;
        \item[(2)] For all $K\geq0$, it holds that
        \begin{equation}\label{eq. Psi_K best estimate}
            \max\{\Psi_{0}(X),K\cdot H(X)\}\leq\Psi_{K}(X)\leq\Psi_{0}(X)+K\cdot H(X) \ \mbox{in}\ Cone_{\Sigma};
        \end{equation}
        \item[(3)] If $0\leq k\leq K$ , then
        \begin{equation*}
            0\leq\Psi_{K}(X)-\Psi_{k}(X)\leq(K-k)\cdot H(X) \ \mbox{in}\ Cone_{\Sigma}.
        \end{equation*}
    \end{itemize}
\end{theorem}
Since $\alpha<\phi$, estimate \eqref{eq. Psi_K best estimate} implies
\[
    \frac{\Psi_K(r\theta)}{H(r\theta)}\longrightarrow K
    \qquad \text{as }r\to\infty,
\]
for each $\theta\in\Sigma$. Therefore, $K$ is the asymptotic harmonic slope of $\Psi_K$. Besides, since $\Psi_{0}(X)\leq C|X|^{\alpha}\ll R^{\phi}$ in $B_{R}$, one has \eqref{eq. K means max in B_R}. This gives an intrinsic meaning to the parameter even though $f$ need not be homogeneous and no explicit formula is available.

The third result affirms that there are no additional global solutions.
\begin{theorem}[Classification]\label{thm. classification}
    Under the same assumptions in Theorem~\ref{thm. construction}, 
    %for any $K\geq0$, the functions $\Psi_{K}$'s constructed in Theorem~\ref{thm. construction} are unique subject to their corresponding inequality requirements. Moreover, 
    any global solution to \eqref{eq. main} must be one of the $\Psi_{K}$'s constructed in Theorem~\ref{thm. construction}.
\end{theorem}
Theorems~\ref{thm. non-existing}--\ref{thm. classification} give a complete answer to the classification problem. The condition $\alpha<\phi$ is necessary and sufficient for global solvability, and the solution set in this range is a one-parameter ordered family. The threshold is geometric: a narrower cone has a larger first spherical eigenvalue and hence a larger characteristic exponent, while a wider cone has a smaller one. These main results therefore identify precisely how the strength of the singularity and the opening of the cone interact.

We emphasize the relation with the existing literature. The local Lipschitz-domain theory in our previous work \cite{GLZ25} supplies the boundary estimates that make the construction possible. The half-space results of \cite{MMS24b,MMS24a} are recovered when $f\equiv1$ and $\phi=1$. Our result goes beyond that setting in two directions. It treats general Lipschitz epigraphical cones, and allows a non-constant coefficient $f(X)$. In particular, the classification does not rely on tangential translations, moving planes, or a reduction to an ordinary differential equation.

An immediate geometric consequence of Theorem \ref{thm. classification} is that the global solutions are independent of the choice of vertex when the cone has multiple vertices or a flat edge. Consider, for example, the cone
\begin{equation*}
    \{(x,y,z)\in\mathbb{R}^{3}:z\geq|x|\},
\end{equation*}
then it follows from Theorem~\ref{thm. classification} that the global solutions $\Psi_{K}$'s do not depend on how we choose the vertex of the cone. In particular, this observation leads to the following one-dimensional symmetry result, which is a generalization of \cite{MMS24b,MMS24a}.
\begin{corollary}\label{cor. great extension, a glimpse}
    Let $f(X)=f(x_{n})$ satisfy $f\in[\lambda,\Lambda]$ and be locally Dini-continuous. Then, when $\gamma>1$, global solutions $u(X):\mathbb{R}^{n}_{+}\to[0,\infty)$ to the problem
    \begin{equation}\label{eq. translation invariant problem}
        \left\{\begin{aligned}
            &-\Delta u=f(x_{n})u^{-\gamma}&\mbox{in }&\mathbb{R}^{n}_{+},\\
            &u=0&\mbox{on }&\partial\mathbb{R}^{n}_{+}
        \end{aligned}\right.
    \end{equation}
    must be one-dimensional, i.e.: $u(X)=u(x_{n})$. On the other hand, when $0<\gamma\leq1$, then the problem \eqref{eq. translation invariant problem} admits no global solutions.
\end{corollary}
\begin{proof}
    It is a direct corollary of Theorem~\ref{thm. non-existing}, Theorem~\ref{thm. construction}, Theorem~\ref{thm. classification}, as well as the translation invariant nature of \eqref{eq. translation invariant problem}. We omit the details.
\end{proof}

\subsection{Main ideas}

Two issues are specific to the global classification problem in an unbounded Lipschitz cone. First, no growth assumption is imposed at infinity. We must prove that an arbitrary global solution has at most the harmonic growth $|X|^{\phi}$. Second, after obtaining this bound, we must exclude all solutions other than the constructed family $\{\Psi_K\}$.

The argument has four main steps.

\begin{itemize}
%[label=\textup{(\roman*)},leftmargin=2.5em]
    \item \emph{The solvability threshold.} We compare a hypothetical global solution with solutions to a constant-coefficient problem in truncated cones. Scaling the truncated problem and using the local lower bound estimate (Lemma~\ref{lem. non degenerate}) shows that the value at a fixed interior point tends to infinity when $\alpha\geq\phi$. This proves Theorem~\ref{thm. non-existing}.

    \item \emph{Construction of the model family.} In the range $\alpha<\phi$, increasing bounded-domain approximation produces the ground solution $\Psi_0$. For $K>0$, we solve truncated problems with boundary values $\Psi_0+K H$. Using the maximum principle provides the two-sided estimate similar to \eqref{eq. Psi_K best estimate}, which is uniform on any fixed compact subset and allows passage to a global solution $\Psi_{K}$. This construction works without any symmetry of $f$.

    \item \emph{From local information to a global growth bound.} Let $h_R$ be the harmonic replacement of an arbitrary global solution $u$ in a large truncated cone. Comparison gives
    \[
        h_R\leq u\leq h_R+\Psi_0.
    \]
    Using Kemper's boundary Harnack principle \cite{K72}, one approximates $h_R$ with a constant multiple of $H$. Evaluating at a fixed interior point controls its coefficient uniformly in $R$. This yields
    \[
        \sup_{B_R\cap\operatorname{Cone}_{\Sigma}}u\leq C(u)R^{\phi}
    \]
    and then $u\leq\Psi_{\overline{K}}$ for some $\overline{K}$, see Lemma~\ref{lem. growth rate at infinity} and Corollary~\ref{cor. bounded by psi K}.

    \item \emph{Reduction of the slope interval.} Once
    \[
        \Psi_k\leq u\leq\Psi_K
    \]
    holds on a large scale, we prove that the interval $[k,K]$ can be shortened by a fixed factor on a smaller scale (Lemma~\ref{lem. reduction of oscillation}). Such an ``oscillation reduction" argument is based on a nonlinear variant of the boundary Harnack principle (Lemma~\ref{lem. nonlinear BHP, if case 1 holds} and \ref{lem. nonlinear BHP, if case 2 holds}). In fact, by analyzing the linearized equation of \eqref{eq. SLEF}, we obtain an interior Harnack estimate in the suspended cylinder. Then, a boundary positivity argument improves either the lower parameter or the upper parameter. Iteration forces the limiting interval to collapse to one point. This proves Theorem~\ref{thm. classification}. 
\end{itemize}

The last two steps form the main highlight of the paper. They replace the translation and one-dimensional tools available in the half-space by a method based on harmonic replacement, boundary Harnack principle, and quantitative reduction of oscillation. The resulting classification is stable under variable coefficients and nonsmooth conical geometry. It also identifies all global cone profiles that can arise in boundary blow-up arguments for \eqref{eq. SLEF}.

\textbf{Organization of the paper.} The rest of this paper is organized as follows. In Section~2, we collect several preliminary results used throughout the paper and prove the nonexistence theorem (Theorem \ref{thm. non-existing}) for the case $\alpha\geq \phi$. In Section 3, we construct global solutions to \eqref{eq. main} in the regime $\alpha< \phi$ (Theorem \ref{thm. construction}), namely the ground solution  $\Psi_0$ and the family $\{\Psi_K\}_{K > 0}$ with ``asymptotic slope" $K$. In Section 4, we establish a global growth rate estimate for arbitrary solutions (Lemma \ref{lem. growth rate at infinity}) and then show that every solution lies below some $\Psi_{\overline{K}}$ (Corollary~\ref{cor. bounded by psi K}). In Section 5, by developing a nonlinear version of the boundary Harnack principle (Lemma \ref{lem. nonlinear BHP, if case 1 holds} and Lemma \ref{lem. nonlinear BHP, if case 2 holds}), we obtain the oscillation reduction estimate (Lemma \ref{lem. reduction of oscillation}), and thereby complete the proof of the classification theorem (Theorem \ref{thm. classification}).

%Section 4 establishes a global growth rate estimate for arbitrary solutions (Lemma \ref{lem. growth rate at infinity}) and then shows that every solution lies below some $\Psi_{\overline{K}}$. This result serves as a key role in deriving the oscillation reduction estimate (Lemma \ref{lem. reduction of oscillation}) established in Section 5, where we develop a nonlinear version of the boundary Harnack principle (Lemma \ref{lem. nonlinear BHP, if case 1 holds} and Lemma \ref{lem. nonlinear BHP, if case 2 holds}) as the central tool, and thereby complete the proof of the classification theorem (Theorem \ref{thm. classification}).
\section{Preliminaries}
\subsection{Review of several known results}
First we recall the following pointwise lower bound estimate in \cite[Lemma 2.2]{GLZ25}. 
\begin{lemma}[Pointwise lower bound]\label{lem. non degenerate}
    Assume that $u\geq0$ satisfies $-\Delta u\geq\lambda u^{-\gamma}$ in $B_{r}$, then
    \begin{equation*}
        u(0)\geq c(n,\gamma,\lambda)r^{\alpha}.
    \end{equation*}
\end{lemma}
Next we provide a simplified version of \cite[Theorem 1.2]{GLZ25}. 
\begin{lemma}[Growth rate estimates]\label{lem. growth rate near the boundary}
    Assume that $\lambda\leq f(X)\leq \Lambda$, and $u$ satisfies
    \begin{equation}\label{eq. Dirichlet problem in B_2}
    \left\{\begin{aligned}
        &-\Delta u=f(X) u^{-\gamma}&\mbox{ in }&Cone_{\Sigma}\cap B_{2},\\
        &u=0&\mbox{ on }&\partial Cone_{\Sigma}\cap B_{2}.
    \end{aligned}
    \right.
\end{equation}where ${Cone_{\Sigma}}=\{x_{n}>g(x')\}$ is a Lipschitz epigraphical cone with $[g]_{C^{0,1}(\mathbb{R}^{n-1})}\leq L$, then there exists $C=C(n,L,\gamma,\lambda,\Lambda,\phi)>0$ such that the following estimates hold:
    \begin{itemize}
        \item[(a)] If $\alpha<\phi$, then
        \begin{equation*}
            u(X)\leq C\|u\|_{L^{\infty}(Cone_{\Sigma}\cap B_2)}\cdot|X|^{\alpha},\ \forall X\in Cone_{\Sigma}\cap B_1 ;
        \end{equation*}
        \item[(b)] If  $\alpha>\phi$, then
        \begin{equation*}
            u(t\vec{e_{n}})\geq \frac{1}{C} t^{\phi}\mbox{ for }t\in[0,1].
        \end{equation*}
        \item[(c)] If  $\alpha=\phi$, then
        \begin{equation*}
            u(t\vec{e_{n}})\geq \frac{1}{C} t^{\phi}(\ln{\frac{1}{t}})^{\phi/2}\mbox{ for }t\in[0,1].
        \end{equation*}
    \end{itemize}
\end{lemma}
Similar to the proof of  \cite[Corollary 1.1]{GLZ25}, we have the following boundary H\"older regularity.
\begin{lemma}[Boundary H\"older regularity]\label{thm.bdy holder}
Under the same assumptions as in Lemma \ref{lem. growth rate near the boundary}, there exist $\mu=\mu(n,\gamma, L)>0$
and $C=C(n,L,\gamma,\lambda,\Lambda)>0$ such that
\begin{equation*}
    \|u\|_{C^{\mu}(Cone_{\Sigma}\cap B_1)}\leq C\|u\|_{L^{\infty}(Cone_{\Sigma}\cap B_2)}.
\end{equation*}
\end{lemma}
In \cite{K72}, Kemper developed the boundary Harnack principle, which states as follows:
\begin{lemma}\label{lem. classical boundary Harnack}
    Let $\Gamma=\{x_{n}=g(x')\}$ be a Lipschitz graph passing through the origin. Let $u,v>0$ be two harmonic functions defined in $\mathcal{GC}_{1}$ such that $u=v=0$ on $\Gamma$. Then there exist constants $C,\epsilon$ depending only on $(n,\|g\|_{C^{0,1}})$ such that
    \begin{equation*}
        \|\frac{u}{v}\|_{C^{\epsilon}(\mathcal{GC}_{1/2})}\leq C\frac{u}{v}(\frac{\vec{e}_{n}}{2}).
    \end{equation*}
\end{lemma}
\begin{remark}\label{rmk. boundary harnack remark}
    It could also be inferred from Lemma~\ref{lem. classical boundary Harnack} that $\displaystyle\max_{\mathcal{GC}_{1/2}}\frac{u}{v}\leq C\min_{\mathcal{GC}_{1/2}}\frac{u}{v}$.
\end{remark}

\subsection{Nonexistence}
 In this subsection, we prove the nonexistence of solutions to \eqref{eq. main} when $\alpha\geq\phi$.
\begin{proof}[Proof of Theorem \ref{thm. non-existing}]
We use the contradiction argument and suppose that $u$ is a global solution to \eqref{eq. main}.
By the well-posedness of \eqref{eq. SLEF} in a bounded Lipschitz region (see \cite[Theorem 1.1]{GLZ25}), we can define $v_R$ to be the solution to 
 \begin{equation*}
    \left\{\begin{aligned}
        &-\Delta v_R=\lambda v_R^{-\gamma}&\mbox{ in }&Cone_{\Sigma}\cap B_{R},\\
        &v_R=0&\mbox{ on }&\partial (Cone_{\Sigma}\cap B_{R}).
    \end{aligned}
    \right.
\end{equation*}  
Since $f(X)\geq \lambda$ in $Cone_{\Sigma}\cap B_{R}$ and $u\geq 0$ on $\partial(Cone_{\Sigma}\cap B_{R})$, then by the maximum principles in \cite[Lemma 2.1]{GLZ25}, it follows that
\[u\geq v_R \ \mbox{in}\ Cone_{\Sigma}\cap B_{R}.\]
Furthermore,  by the scaling-invariance property   $v_R(X)=R^{\alpha}\cdot v_1(\frac{X}{R})$, we have 
 \[u(\vec{e_n})\geq v_R(\vec{e_n})\geq R^{\alpha}\cdot v_1(\frac{\vec{e_n}}{R})\ \mbox{for all}\ R>1.\]
Since $\alpha\geq\phi$,  it follows from (b) and (c) of Lemma \ref{lem. growth rate near the boundary} that the above lower bound tends to infinity as $R\to \infty,$ which is a contradiction. Therefore, we verify the nonexistence of global solutions to \eqref{eq. main} when $\alpha\geq \phi$.
\end{proof}

\section{Construction of global solutions}
In this section we construct global solutions to \eqref{eq. main} when $\alpha<\phi$ and thus prove Theorem~\ref{thm. construction}.
\subsection{The ground solution \texorpdfstring{$\Psi_{0}$}{Lg}}
In this part, we construct the solution $\Psi_{0}$ in Theorem~\ref{thm. construction} (1) and prove that all global solutions to \eqref{eq. main} are no less than $\Psi_{0}$.
\begin{proof}[Proof of Theorem~\ref{thm. construction} (1)]
For any $R>0$, let $\Psi_{0,R}\geq 0$ be the solution of 
     \begin{equation}\label{eq. Def Psi_0,R}
    \left\{\begin{aligned}
        &-\Delta \Psi_{0,R}=f(X)\cdot \Psi_{0,R}^{-\gamma}&\mbox{ in }&Cone_{\Sigma}\cap B_{R},\\
        &\Psi_{0,R}=0&\mbox{ on }&\partial (Cone_{\Sigma}\cap B_{R}).
    \end{aligned}
    \right.
\end{equation} 
The existence of $\Psi_{0,R}$ is guaranteed by \cite[Theorem 1.1 \& Remark 1.3]{GLZ25}. Clearly, $\Psi_{0,R}$ (extended trivially outside $Cone_{\Sigma}\cap B_{R}$) is an increasing sequence in $R$. Let $V_R$ be the solution of   \eqref{eq. Def Psi_0,R} with $f$ being replaced by $\Lambda,$ then 
\[V_R\geq  \Psi_{0,R} \mbox{ in } Cone_{\Sigma}\cap B_{R}.\]
For every fixed $r$ and $R\geq 2r,$ it follows from the scaling-invariance property $V_R(X)=R^{\alpha}V_1(\frac{X}{R})$ that
\begin{equation*}
    \|\Psi_{0,R}\|_{L^{\infty}(Cone_{\Sigma}\cap B_{r})}\leq R^{\alpha}\|V_{1}\|_{L^{\infty}(Cone_{\Sigma}\cap B_{r/R})}.
\end{equation*}
Applying Lemma~\ref{lem. growth rate near the boundary} (a) to $V_1$ gives that
\begin{equation}\label{psi infty}
    \|\Psi_{0,R}\|_{L^{\infty}(Cone_{\Sigma}\cap B_{r})}\leq C r^{\alpha}\|V_{1}\|_{L^{\infty}(Cone_{\Sigma}\cap B_{1})}=:Cr^\alpha.
\end{equation}
Then, by the monotonicity and the local bound \eqref{psi infty}, we may define 
\[\Psi_{0}(X):=\lim_{R\to\infty}\Psi_{0,R}(X) \ \mbox {for}\ X\in Cone_{\Sigma}.\]
Obviously, one has $\Psi_0(X)\leq C|X|^{\alpha}.$

It remains to show  that $\Psi_0$ is a global classical solution to \eqref{eq. main}.
Using Lemma \ref{thm.bdy holder} to $\Psi_{0,R}$ and \eqref{psi infty}, there exists a constant $\mu>0$ such that
\begin{equation}\label{psi holder}
    [\Psi_{0,R}]_{C^{\mu}(Cone_{\Sigma}\cap B_{r/2})} \leq C r^{-\mu}\|\Psi_{0,R}\|_{L^{\infty}(Cone_{\Sigma}\cap B_{r})}\leq C r^{\alpha-\mu}.
\end{equation}
Hence, $\Psi_0$ is locally  H\"older continuous and $$\Psi_0=0\  \mbox{on}\  \partial Cone_\Sigma.$$ 
Furthermore, for every $E\subset\subset Cone_{\Sigma}$, Lemma~\ref{lem. non degenerate} gives a uniform lower bound of $\Psi_{0,R}$ as $R\to\infty$. Together with \eqref{psi holder} and that $f(X)\in C^{Dini}_{loc}$, one gets a uniform local Dini modulus for $f(X)\cdot\Psi_{0,R}^{-\gamma}$. Then, the standard interior Schauder estimate implies that
\begin{equation*}
    \|\Psi_{0,R}\|_{C^{2}_{\sigma}(E)}\leq C(E),\quad\mbox{for every }E\subset\subset Cone_\Sigma\mbox{ and sufficiently large }R.
\end{equation*}
Here, $C^{2}_{\sigma}$ means that the second derivative of $\Psi_{0,R}$ is equicontinuous with respect to a certain modulus-of-continuity $\sigma$, i.e.
\begin{equation*}
    \Big|D^{2}\Psi_{0,R}(x)-D^{2}\Psi_{0,R}(y)\Big|\leq C(E)\cdot\sigma(|x-y|),\quad\mbox{for all }x,y\in E.
\end{equation*}
By letting $R\to\infty$, we conclude that $\Psi_0$ is a classical solution of \eqref{eq. main} and complete the proof  of Theorem \ref{thm. construction}(1).
\end{proof}

A by-product is that any global solution is greater than $\Psi_{0}$.
\begin{lemma}\label{lem. psi 0 is the lower bound}
    Let $\alpha<\phi$ and let $u$ be a global solution to \eqref{eq. main}. Then $u\geq\Psi_{0}$ everywhere.
\end{lemma}
\begin{proof}
    As $\Psi_{0,R}$ vanishes (then less than $u$) on $\partial(B_{R}\cap Cone_{\Sigma})$, it is then less than $u$ inside $B_{R}\cap Cone_{\Sigma}$ by the maximum principle. Therefore,
    \begin{equation*}
u(X)\geq\lim_{R\to\infty}\Psi_{0,R}(X)=\Psi_{0}(X).
    \end{equation*}
    for any fixed $X\in Cone_{\Sigma}$.
\end{proof}
\subsection{Solutions with ``asymptotic slope" \texorpdfstring{$K$}{Lg}}
In this part, we construct $\Psi_{K}$ for all $K>0$.
\begin{proof}[Proof of Theorem \ref{thm. construction} (2)]
Again, by \cite[Theorem 1.1 \& Remark 1.3]{GLZ25}, we define a sequence of functions $\Psi_{K,R}(X)$ satisfying
    \begin{equation}\label{eq. slope solution in bounded interval}
        \left\{\begin{aligned}
            &-\Delta\Psi_{K,R}(X)=\frac{f(X)}{\Psi_{K,R}(X)^{\gamma}}&\mbox{in }&Cone_{\Sigma}\cap B_R,\\
            &\Psi_{K,R}(X)=\Psi_{0}(X)+K\cdot H(X)&\mbox{on }&\partial(Cone_{\Sigma}\cap B_{R}).
        \end{aligned}\right.
    \end{equation}
    We further define $\Psi_{K,R}=\Psi_{0}(X)+K\cdot H(X)$ in $B_{R}^{c}\cap Cone_{\Sigma}$, so that $\Psi_{K,R}$ becomes a continuous function in $Cone_{\Sigma}$.
    
    We claim that
    \begin{equation}\label{eq. bounded slope solution also bounded}
        \max\{KH(X),\Psi_{0}(X)\}\leq \Psi_{K,R}(X)\leq KH(X)+\Psi_{0}(X)\  \mbox{for}\ X\in B_{R}\cap Cone_{\Sigma}.
    \end{equation}
    To see \eqref{eq. bounded slope solution also bounded}, we use the maximum principle in the following way:
    \begin{itemize}
        \item Since $K\cdot H(X)$ is harmonic in $Cone_\Sigma$, it is a sub-solution to \eqref{eq. slope solution in bounded interval} with boundary data smaller than $\Psi_{K,R}$, so $KH(X)\leq \Psi_{K,R}(X)$ in $Cone_{\Sigma}\cap B_{R}$.
        \item Since $\Psi_{0}\leq\Psi_{K,R}$ on $\partial(B_{R}\cap Cone_{\Sigma})$ and they both satisfy $-\Delta u=f(X)\cdot u^{-\gamma}$ in $B_{R}\cap Cone_{\Sigma}$, one has $\Psi_{0}\leq\Psi_{K,R}$ in $B_{R}\cap Cone_{\Sigma}$.
        \item It follow from
    \begin{equation*}
        -\Delta(\Psi_{0}(X)+K\cdot H(X))=\frac{f(X)}{\Psi_{0}(X)^{\gamma}}\geq\frac{f(X)}{(\Psi_{0}(X)+KH(X))^{\gamma}}
    \end{equation*}
    that $\Psi_{0}(X)+KH(X)$ is a super-solution to \eqref{eq. slope solution in bounded interval} with the same boundary data as $\Psi_{K,R}$, so $\Psi_{K,R}(X)\leq KH(X)+\Psi_{0}(X)$ in $B_{R}\cap Cone_{\Sigma}$.
    \end{itemize}
    Then, \eqref{eq. bounded slope solution also bounded} is verified.   
    
   We next show that $\Psi_{K,R}(X)$ is decreasing in $R$. In fact, we arbitrarily choose $R_{1}<R_{2}$, then since $\Psi_{K,R_{2}}(X)\leq K\cdot H(X)+\Psi_{0}(X)$, we see that $\Psi_{K,R_{1}}(X)$ and $\Psi_{K,R_{2}}(X)$ satisfy the same equation in $Cone_{\Sigma}\cap B_{R_1}$, but the boundary condition of $\Psi_{K,R_{1}}(X)$ is larger than that of $\Psi_{K,R_{2}}(X)$. By the maximum principle, we have $\Psi_{K,R_{1}}(X)\geq \Psi_{K,R_{2}}(X)$ in $Cone_{\Sigma}\cap B_{R_1}$ (and also in $Cone_{\Sigma}$). 

     From the bound \eqref{eq. bounded slope solution also bounded} and the monotonicity of $\Psi_{K,R}(X)$ in $R$,  one can define 
    \begin{equation*}
         \Psi_{K}(X):=\lim_{R\to \infty}\Psi_{K,R}(X),\ \mbox{for}\  X\in Cone_{\Sigma},
    \end{equation*}
     which also satisfies the estimate \eqref{eq. bounded slope solution also bounded}. In particular,
     \begin{equation*}
         \lim_{X\to Y}\Psi_{K}(X)=0,\quad\mbox{for all }Y\in\partial Cone_{\Sigma}.
     \end{equation*}
     Arguing as in the construction of $\Psi_0$,  we derive that $\Psi_{K,R}$ is locally $C^{2}_{\sigma}$ with norms uniformly bounded in $R$. Consequently, $\Psi_{K}(X)$ is a classical global solution to \eqref{eq. main}. The upper and lower bounds of $\Psi_{K}$ follow from taking the limit of \eqref{eq. bounded slope solution also bounded}.
\end{proof}

Finally, we give the proof of Theorem~\ref{thm. construction} (3). We remark that this implies that $\Psi_{K}(X)$ is continuous and increasing with respect to the parameter $K$ for every fixed $X\in Cone_{\Sigma}$.
\begin{proof}[Proof of Theorem~\ref{thm. construction} (3)]
    For every $R>0$, one has
    \begin{equation*}
        \left\{\begin{aligned}
            &\Psi_{K,R}-\Psi_{k,R}=(K-k)\cdot H&\mbox{ on }&\partial(B_{R}\cap Cone_{\Sigma}),&\mbox{if }&k>0,\\
            &\Psi_{K,R}-\Psi_{k}=(K-k)\cdot H&\mbox{ on }&\partial(B_{R}\cap Cone_{\Sigma}),&\mbox{if }&k=0.
        \end{aligned}\right.
    \end{equation*}
    Since $\Psi_{0}$, $\Psi_{k,R}$, $\Psi_{K,R}$ all satisfy \eqref{eq. SLEF} in $B_{R}\cap Cone_{\Sigma}$, it follows from \cite[Lemma 2.1]{GLZ25} that
    \begin{equation*}
        \left\{\begin{aligned}
            &\Psi_{K,R}\geq\Psi_{k,R}&\mbox{ in }&B_{R}\cap Cone_{\Sigma},&\mbox{if }&k>0,\\
            &\Psi_{K,R}\geq\Psi_{k}&\mbox{ in }&B_{R}\cap Cone_{\Sigma},&\mbox{if }&k=0.
        \end{aligned}\right.
    \end{equation*}
    Moreover, since $f(X)u^{-\gamma}$ is decreasing in $u$, the estimate above implies that
    \begin{equation*}
        \left\{\begin{aligned}
            &\Delta(\Psi_{K,R}-\Psi_{k,R})\geq0&\mbox{ in }&B_{R}\cap Cone_{\Sigma},&\mbox{if }&k>0,\\
            &\Delta(\Psi_{K,R}-\Psi_{k})\geq0&\mbox{ in }&B_{R}\cap Cone_{\Sigma},&\mbox{if }&k=0.
        \end{aligned}\right.
    \end{equation*}
    Since $(K-k)\cdot H(X)$ is harmonic, it follows that
    \begin{equation*}
        \left\{\begin{aligned}
            &0\leq\Psi_{K,R}-\Psi_{k,R}\leq(K-k)\cdot H(X)&\mbox{ in }&B_{R}\cap Cone_{\Sigma},&\mbox{if }&k>0,\\
            &0\leq\Psi_{K,R}-\Psi_{k}\leq(K-k)\cdot H(X)&\mbox{ in }&B_{R}\cap Cone_{\Sigma},&\mbox{if }&k=0.
        \end{aligned}\right.
    \end{equation*}
    Sending $R\to\infty$ yields the desired estimate of Theorem~\ref{thm. construction} (3).
\end{proof}

\section{Growth rate estimate}
The first key step in proving Theorem~\ref{thm. classification} is to estimate the growth rate of a global solution to \eqref{eq. main} at infinity.
\begin{lemma}\label{lem. growth rate at infinity}
    Assume that $\alpha<\phi$, $0<\lambda\leq f(X)\leq\Lambda$, and let $u$ be a global solution to \eqref{eq. main}. Then there exists a non-universal constant $C(u)$ such that
    \begin{equation*}
        \max_{B_{R}\cap Cone_{\Sigma}}u(X)\leq C(u)\cdot R^{\phi},\quad\mbox{for all }R\geq2.
    \end{equation*}
\end{lemma}
\begin{proof}
    For every $R\geq2$, let $h_{R}$ be the harmonic replacement of $u$ in $B_{2R}\cap Cone_{\Sigma}$, i.e.
        \begin{equation*}
            \left\{\begin{aligned}
            &-\Delta h_{R}(X)=0&\mbox{ in }&B_{2R}\cap Cone_{\Sigma},\\
            &h_{R}(Y)=u(Y)&\mbox{ on }&\partial(B_{2R}\cap Cone_{\Sigma}).
        \end{aligned}\right.
        \end{equation*}

        We first claim that
        \begin{equation}\label{eq. error between u and its harmonic replacement}
            h_{R}(X)\leq u(X)\leq h_{R}(X)+\Psi_{0}(X)\quad\mbox{in }B_{2R}\cap Cone_{\Sigma}.
        \end{equation}
        In fact, since
        \begin{equation*}
            \left\{\begin{aligned}
                &-\Delta u\geq-\Delta h_{R}&\mbox{in }&B_{2R}\cap Cone_{\Sigma},\\
                &u(Y)=h_{R}(Y)&\mbox{on }&\partial(B_{2R}\cap Cone_{\Sigma}),
            \end{aligned}\right.
        \end{equation*}
        the first inequality in \eqref{eq. error between u and its harmonic replacement}  follows from the maximum principle. For the second inequality, observe that
 \begin{equation*}
            \left\{\begin{aligned}
                &-\Delta(h_{R}+\Psi_{0})=f(X)\Psi_{0}(X)^{-\gamma}>f(X)\cdot(h_{R}+\Psi_{0})^{-\gamma}&\mbox{in }&B_{2R}\cap Cone_{\Sigma},\\
                &u(Y)\leq h_{R}(Y)+\Psi_{0}(Y)&\mbox{on }&\partial(B_{2R}\cap Cone_{\Sigma}),
            \end{aligned}\right.
        \end{equation*}
The comparison principle for \eqref{eq. SLEF} (see \cite[Lemma 2.1]{GLZ25}) then yields the second inequality in  \eqref{eq. error between u and its harmonic replacement}. 

        Next, by the boundary Harnack principle for the Laplace equation (see Lemma~\ref{lem. classical boundary Harnack} and Remark~\ref{rmk. boundary harnack remark}), there exist a universal constant $C_{2}$ depending only on the domain $Cone_{\Sigma}$ and some $L_{R}>0$ depending on $u$ and $R$, such that
        \begin{equation}\label{eq. BHP for h_R}
            L_{R}H(X)\leq h_{R}(X)\leq C_{2}L_{R}H(X)\quad\mbox{in }B_{R}\cap Cone_{\Sigma}.
        \end{equation}

        Taking $X=\vec{e}_{n}\subseteq B_{R}\cap Cone_{\Sigma}$ in \eqref{eq. error between u and its harmonic replacement} and \eqref{eq. BHP for h_R} yields
        \begin{equation*}
            L_{R}H(\vec{e}_{n})\leq h_{R}(\vec{e}_{n})\leq u(\vec{e}_{n}),\quad\mbox{for all }R\geq2.
        \end{equation*}
        This gives a uniform upper bound for $L_{R}$, which together with \eqref{eq. error between u and its harmonic replacement} and \eqref{eq. BHP for h_R} gives that
        \begin{equation*}
            u(X)\leq C_{2}L_{R}H(X)+\Psi_{0}(X)\leq C_{2}\frac{u(\vec{e}_{n})}{H(\vec{e}_{n})}H(X)+\Psi_{0}(X)\quad\mbox{in }B_{R}\cap Cone_{\Sigma}.
        \end{equation*}
        Notice that by Definition~\ref{def. H(X)}, Theorem~\ref{thm. construction} (1), and the assumption $\alpha<\phi$, one has
        \begin{equation*}
            \|H\|_{L^{\infty}(B_{R}\cap Cone_{\Sigma})}=R^{\phi}\quad\mbox{and }\|\Psi_{0}\|_{L^{\infty}(B_{R}\cap Cone_{\Sigma})}\leq C_{0}R^{\alpha}\leq C_{0}R^{\phi},
        \end{equation*}
        it then follows that
        \begin{equation*}
            u(X)\leq\Big(C_{2}\frac{u(\vec{e}_{n})}{H(\vec{e}_{n})}+C_{0}\Big)\cdot R^{\phi}=:C(u)\cdot R^{\phi}\quad\mbox{in }B_{R}\cap Cone_{\Sigma}.
        \end{equation*}
        Hence, the proof of Lemma~\ref{lem. growth rate at infinity} is completed.
\end{proof}

As a consequence, we see that a global solution $u$ must be bounded from above by some $\Psi_{\overline{K}}$ everywhere in $Cone_{\Sigma}$, where $\overline{K}$ is non-universal constant depending on $u$.
\begin{corollary}\label{cor. bounded by psi K}
    Assume that $\alpha<\phi$, $f(X)$ is locally Dini continuous in $Cone_{\Sigma}$, with $0<\lambda\leq f(X)\leq\Lambda$. Let $u$ be a global solution to \eqref{eq. main}, then there exists a constant $\overline{K}=\overline{K}(u)>0$ such that \[u(X)\leq\Psi_{\overline{K}}(X)\ \mbox{in}\  Cone_{\Sigma}.\] 
\end{corollary}
\begin{proof}
 For any $R\geq 2$, let $h_{R}$ denote the harmonic replacement of $u$ in $B_{R}\cap Cone_{\Sigma}$, defined by
    \begin{equation*}
        \left\{\begin{aligned}
            &-\Delta h_{R}(X)=0&\mbox{ in }&B_{R}\cap Cone_{\Sigma},\\
            &h_{R}(Y)=u(Y)&\mbox{ on }&\partial(B_{R}\cap Cone_{\Sigma}).
        \end{aligned}\right.
    \end{equation*}
    Analogous  to the argument in Lemma~\ref{lem. growth rate at infinity}, there holds
\begin{equation}\label{est-u}
         h_{R}(X)\leq u(X)\leq h_{R}(X)+\Psi_{0}(X)\ \mbox{in}\ B_{R}\cap Cone_{\Sigma}.
    \end{equation}
      Moreover, by Lemma~\ref{lem. growth rate at infinity}, for sufficiently large $R$, we have
    \begin{equation*}
        h_R(\frac{R\vec{e}_{n}}{2})\leq u(\frac{R\vec{e}_{n}}{2})\leq C_{1}(u)\cdot R^{\phi},
    \end{equation*}
  where  $C_{1}(u)>0$ depends on 
$u$ but is independent of 
$R$. 
Since  $H(\frac{R\vec{e}_{n}}{2})\sim R^{\phi}$, applying the rescaled version of
%\[\frac{h_R(\frac{R\vec{e}_{n}}{2})}{H(\frac{R\vec{e}_{n}}{2})}\leq C_1(u)\]
  the boundary Harnack principle (see Lemma~\ref{lem. classical boundary Harnack} and Remark~\ref{rmk. boundary harnack remark}) yields
  \begin{equation}\label{est-h}
   h_{R}(X)\leq C_{2}\cdot C_{1}(u)H(X)\ \mbox{in}\ B_{R/2}\cap Cone_{\Sigma}   
  \end{equation} for some universal constant $C_{2}>0$.

Now set $\overline{K}=C_{2}\cdot C_{1}(u)$, which is independent of $R$. Combining \eqref{est-u} and \eqref{est-h}, we obtain 
    \begin{equation*}
        u(Y)\leq\overline{K}\cdot H(Y)+\Psi_{0}(Y)=\Psi_{\overline{K},R/2}(Y)\quad\mbox{on }\partial(B_{R/2}\cap Cone_{\Sigma}),
    \end{equation*}
    where $\Psi_{\overline{K},R/2}$ is the solution to  the bounded-domain problem \eqref{eq. slope solution in bounded interval}.
    The maximum principle then gives
    \begin{equation*}
        u(X)\leq\Psi_{\overline{K},R/2}(X)\quad\mbox{in }B_{R/2}\cap Cone_{\Sigma}.
    \end{equation*}
    Finally,  it follows from the construction of $\Psi_{K}$ in Theorem~\ref{thm. construction} that
    \begin{equation*}
        u(X)\leq\lim_{R\to\infty}\Psi_{\overline{K},R/2}(X)=\Psi_{\overline{K}}(X)
    \end{equation*}
    for any fixed point $X\in Cone_{\Sigma}$. This completes the proof of Corollary~\ref{cor. bounded by psi K}.
\end{proof}

\section{Reduction of oscillation}
In this section, we finish the proof of Theorem~\ref{thm. classification}. The oscillation reduction estimate stated below is central to our proof. To establish it, we develop a crucial nonlinear variant of the Harnack inequalities on Lipschitz-type domains.

\begin{lemma}\label{lem. reduction of oscillation}
    Assume that $\alpha<\phi$, $0<\lambda\leq f(X)\leq\Lambda$, and $f(X)$ is locally Dini continuous in $Cone_{\Sigma}$. There exist a decreasing function $\underline{R}:(0,+\infty)\to(0,+\infty)$ and a universal constant $\sigma$, such that the following holds: For any $0\leq k<K$, and any $R\geq\underline{R}(K-k)$, if $u$ satisfies
    \begin{equation*}
        -\Delta u=f(X)\cdot u^{-\gamma}\mbox{ and }\Psi_{k}(X)\leq u(X)\leq\Psi_{K}(X)\quad\mbox{in }\mathcal{GC}_{2R},\quad u=0\mbox{ on }\Gamma\cap\mathcal{GC}_{2R},
    \end{equation*}
    then there exists $k'\leq K'$ such that $K'-k'\leq(1-\sigma)(K-k)$ and
    \begin{equation*}
        \Psi_{k'}(X)\leq u(X)\leq\Psi_{K'}(X)\quad\mbox{in }\mathcal{GC}_{R/2}.
    \end{equation*}
    Here, $\mathcal{GC}_{R}$ is the grounded cylinder given in Definition~\ref{def. cylindrical domains}.
\end{lemma}
\subsection{Classification of global solutions}
Let us see how Lemma~\ref{lem. reduction of oscillation} implies Theorem~\ref{thm. classification} and postpone its proof afterwards.
\begin{proof}[Proof of Theorem~\ref{thm. classification}]
    By Lemma~\ref{lem. psi 0 is the lower bound} and Corollary~\ref{cor. bounded by psi K}, there exists some $\overline{K}$ depending on $u$ such that
    \begin{equation*}
        \Psi_{0}(X)\leq u(X)\leq\Psi_{\overline{K}}(X)\quad\mbox{in }Cone_{\Sigma}.
    \end{equation*}
    For every $R>0$, define
    \begin{equation*}
        k_{R}=\sup\big\{t\in[0,\overline{K}]:u\geq\Psi_{t}\mbox{ in }\mathcal{GC}_{R}\big\},\quad K_{R}=\inf\big\{t\in[0,\overline{K}]:u\leq\Psi_{t}\mbox{ in }\mathcal{GC}_{R}\big\}.
    \end{equation*}
    Clearly, one has
    \begin{equation*}
        k_{R_{1}}\geq k_{R_{2}}\mbox{ and }K_{R_{1}}\leq K_{R_{2}}\quad\mbox{for all }R_{1}\leq R_{2}.
    \end{equation*}
    Moreover, by Theorem~\ref{thm. construction} (3), one has $0\leq k_{R}\leq K_{R}\leq\overline{K}$ and
    \begin{equation*}
        \Psi_{k_{R}}(X)\leq u(X)\leq\Psi_{K_{R}}(X)\mbox{ in }\mathcal{GC}_{R}.
    \end{equation*}
    
    Then, we denote
    \begin{equation*}
        k_{\infty}=\lim_{R\to\infty}k_{R},\quad K_{\infty}=\lim_{R\to\infty}K_{R}.
    \end{equation*}
    Clearly, $0\leq k_{\infty}\leq K_{\infty}\leq\overline{K}$. Using Theorem~\ref{thm. construction} (3) again, we have
    \begin{equation*}
        \Psi_{k_{\infty}}(X)\leq u(X)\leq\Psi_{K_{\infty}}(X)\quad\mbox{in }Cone_{\Sigma}.
    \end{equation*}

    It suffices to show $k_{\infty}=K_{\infty}$. Suppose that the identity $k_{\infty}=K_{\infty}$ fails, i.e. $k_{\infty}<K_{\infty}$, then there exists some $\epsilon>0$ and $R_{0}>0$, such that
    \begin{equation*}
        K_{R}-k_{R}\geq\epsilon,\quad\mbox{for all }R\geq R_{0}.
    \end{equation*}
    By apply Lemma~\ref{lem. reduction of oscillation} with $k=k_{2R}$ and $K=K_{2R}$, one has that
    \begin{equation*}
        K_{R/2}-k_{R/2}\leq(1-\sigma)\cdot(K_{2R}-k_{2R}),\quad\mbox{for all }R\geq\max\big\{\underline{R}(\epsilon),\frac{1}{2}R_{0}\big\}.
    \end{equation*}
    Sending $R\to\infty$ implies that
    \begin{equation*}
        K_{\infty}-k_{\infty}\leq(1-\sigma)\cdot(K_{\infty}-k_{\infty}),
    \end{equation*}
    contradicting the hypothesis that $k_{\infty}<K_{\infty}$. This completes the proof of Theorem~\ref{thm. classification}.
\end{proof}
\subsection{Proof of the oscillation reduction property}
In the remaining part of this paper, our job is to prove Lemma~\ref{lem. reduction of oscillation}. We first need the following weaker criteria to show a function is positive.
\begin{lemma}\label{lem. ensure positive}
    Let $\Gamma$ be a Lipschitz graph such that $[g]_{C^{0,1}}\leq L$. There exist constants $(M,\delta)$ depending only on $(n,L,C_{1})$ such that if the following holds:
    \begin{itemize}
        \item[(a)] $\Delta w\leq\frac{C_{1}}{dist(X,\Gamma)^{2}}\max\{w,0\}$ and $w\geq-1$ in the interior of $\mathcal{GC}_{r}$;
        \item[(b)] $w\geq M$ in $\mathcal{SC}_{r}=\mathcal{SC}_{r,\delta}(0)$ (see Definition~\ref{def. cylindrical domains});
        \item[(c)] $w=0$ on $\Gamma$, in the continuous sense or the trace sense.
    \end{itemize}
    Then it is guaranteed that $w\geq0$ in $\mathcal{GC}_{r/2}$.
\end{lemma}
\begin{proof}
    Let $(M,\delta_{0})$ be the constants in \cite[Lemma 3.1]{GLZ25}, and set $\delta=\min\{\frac{\delta_{0}}{2},\frac{1}{10}\}$. For each given $x'\in B_{r/2}'$, it follows from \cite[Lemma 3.1]{GLZ25} that $w\geq0$ on the vertical segment $\{x'\}\times[g(x'),g(x')+\frac{r}{2}]$. By letting $x'$ vary in $B_{r/2}'$, one has that
    \begin{equation*}
        w\geq0\mbox{ in }\bigcup_{x'\in B_{r/2}'}\{x'\}\times[g(x'),g(x')+\frac{r}{2}],
    \end{equation*}
    and the union of the vertical segments includes $\mathcal{GC}_{r/2}$. This completes the proof of Lemma~\ref{lem. ensure positive}.
\end{proof}

In practice, we will choose a specific $C_{1}$ in Lemma~\ref{lem. ensure positive} in the following way: Let $u_{1},u_{2}$ be two global solutions to \eqref{eq. main} in $Cone_{\Sigma}$, then
\begin{equation}\label{eq. lagrange MVT}
    \Delta(u_{1}-u_{2})=f(X)(u_{2}^{-\gamma}-u_{1}^{-\gamma})=\frac{\gamma\cdot f(X)}{\xi^{1+\gamma}}(u_{1}-u_{2}),\quad\xi\mbox{ is between }u_{1}(X)\mbox{ and }u_{2}(X).
\end{equation}
Recall that $\lambda\leq f(X)\leq\Lambda$, and by Lemma~\ref{lem. non degenerate},
\begin{equation*}
    \xi(X)\geq\min\{u_{1}(X),u_{2}(X)\}\geq c(n,\gamma,\lambda)dist(X,\partial Cone_{\Sigma})^{\alpha}\geq c(n,\gamma,\lambda,L)|x_{n}-g(x')|^{\alpha},
\end{equation*}
so there exists some $C_{1}$ such that
\begin{equation}\label{eq. lagrange MVT coefficient -2}
    0\leq\frac{\gamma\cdot f(X)}{\xi^{1+\gamma}}\leq\frac{C_{1}}{dist(X,\Gamma)^{2}}.
\end{equation}
Using such a $C_{1}$, Lemma~\ref{lem. ensure positive} then produces two constants $(\delta,M)$, and we will treat them as fixed constants in the proof of Lemma~\ref{lem. reduction of oscillation}.

Before we finally work on the proof of Lemma~\ref{lem. reduction of oscillation}, we introduce one more constant $\eta$ depending on $\delta$. Let $\Phi(X)$ be the ground solution constructed in Theorem~\ref{thm. construction} (1) for the constant coefficient $f\equiv\Lambda$, i.e.
\begin{equation*}
    \left\{\begin{aligned}
        &-\Delta\Phi=\Lambda\cdot\Phi^{-\gamma}\mbox{ and }\Phi(X)\leq C|X|^{\alpha}&\mbox{in }&Cone_{\Sigma},\\
        &\Phi=0&\mbox{on }&\partial Cone_{\Sigma},
    \end{aligned}\right.
\end{equation*}
We claim that $\Phi(X)$ is homogeneous of degree $\alpha$, i.e.
\begin{equation}\label{eq. PHI is homogeneous}
    \Phi(rX)=r^{\alpha}\Phi(X).
\end{equation}
In fact, let $\Phi_{R}$ denote the solution to $-\Delta\Phi_{R}=\Lambda\cdot\Phi_{R}^{-\gamma}$ in $B_{R}\cap Cone_{\Sigma}$ with zero boundary data. By noticing that $\Phi_{r R}(r X)=r^{\alpha}\Phi_{R}(X)$ for all $X\in B_{R}\cap Cone_{\Sigma}$, taking the limit as $R\to\infty$ verifies \eqref{eq. PHI is homogeneous}.

Besides, recall that in the proof of Theorem~\ref{thm. construction}, $\Psi_{0}$ is the limit of $\Psi_{0,R}$, with
\begin{equation*}
    \left\{\begin{aligned}
    &-\Delta\Psi_{0,R}(X)=f(X)\cdot\Psi_{0,R}(X)^{-\gamma}\leq \Lambda\cdot\Psi_{0,R}(X)^{-\gamma}&\mbox{ in }B_{R}\cap Cone_{\Sigma},\\
    &\Psi_{0,R}(X)=0\leq\Phi(X)&\mbox{ on }\partial (B_{R}\cap Cone_{\Sigma}),
\end{aligned}\right.
\end{equation*}
it then follows that $\Psi_{0,R}(X)\leq\Phi(X)$ in $B_{R}\cap Cone_{\Sigma}$. In conclusion,
\begin{equation}\label{eq. Psi0<=Phi}
    \Psi_{0}(X)\leq\Phi(X),\quad\mbox{for all }X\in Cone_{\Sigma}.
\end{equation}
Now, we define $\eta=\eta(\delta)>0$ to be a constant such that
\begin{equation}\label{eq. choice of eta}
    \min_{X\in\mathcal{SC}_{1,\delta}}H(X)\geq2\eta\max_{Y\in\mathcal{GC}_{1}\setminus\mathcal{SC}_{1,\delta}}H(Y),\quad\min_{X\in\mathcal{SC}_{1,\delta}}\Phi(X)\geq2\eta\max_{Y\in\mathcal{GC}_{1}\setminus\mathcal{SC}_{1,\delta}}\Phi(Y).
\end{equation}

We have the following nonlinear variant of Kemper's boundary Harnack principle.
\begin{lemma}\label{lem. nonlinear BHP, if case 1 holds}
    Let $(\delta,M,\eta)$ be the constants mentioned in Lemma~\ref{lem. ensure positive} and \eqref{eq. choice of eta}. Assume that $\alpha<\phi$, $0<\lambda\leq f(X)\leq\Lambda$, and $f(X)$ is locally Dini continuous in $Cone_{\Sigma}$. There exists a decreasing function $\underline{R}:(0,+\infty)\to(0,+\infty)$, such that the following holds: For any $0\leq k<K$, and any $R\geq\underline{R}(K-k)$, if $u$ satisfies
    \begin{equation*}
        -\Delta u=f(X)\cdot u^{-\gamma}\mbox{ and }u(X)\geq\Psi_{k}(X)\quad\mbox{in }\mathcal{GC}_{2R},\quad u=0\mbox{ on }\Gamma=\partial Cone_{\Sigma},
    \end{equation*}
    and also satisfies
    \begin{equation*}
        u(\vec{P})\geq\frac{\Psi_{k}(\vec{P})+\Psi_{K}(\vec{P})}{2}\quad\mbox{for }\vec{P}=R\vec{e}_{n}.
    \end{equation*}
    Then, there exists a universal constant $C_{1}=C_{1}(n,L,\gamma,\delta,\lambda,\Lambda)$, such that
    \begin{equation}\label{eq. linearized Harnack 1}
        C_{1}^{-1}\cdot\Big(u(\vec{P})-\Psi_{k}(\vec{P})\Big)\leq u(X)-\Psi_{k}(X)\leq C_{1}\cdot\Big(u(\vec{P})-\Psi_{k}(\vec{P})\Big)\mbox{ in }\mathcal{SC}_{R}=\mathcal{SC}_{R,\delta}(0).
    \end{equation}
    Furthermore, denote
    \begin{equation}\label{eq. sigma choice to be verified in step 3}
        k'=k+\sigma(K-k),\quad\mbox{where }\sigma=\frac{1}{4(C_{1})^{2}(\frac{M}{\eta}+1)}.
    \end{equation}
    then one has
    \begin{equation}\label{eq. lower bound update to k'}
        u(X)\geq\Psi_{k'}(X)\mbox{ in }\mathcal{GC}_{R/2}.
    \end{equation}
\end{lemma}
\begin{proof}
    We first verify \eqref{eq. linearized Harnack 1}. Since $u(X)$ and $\Psi_{k}(X)$ both satisfy \eqref{eq. main}, the following linearized equation holds for their difference $w=u-\Psi_{k}\geq0$:
    \begin{equation*}
        \Delta w=f(X)\cdot\gamma\cdot\xi(X)^{-\gamma-1}\cdot w(X),\quad\mbox{where }\Psi_{k}(X)\leq\xi(X)\leq u(X).
    \end{equation*}
    By the lower bound estimate (Lemma~\ref{lem. non degenerate}), one has
    \begin{equation*}
        \xi(X)^{-\gamma-1}\leq\Psi_{k}(X)^{-\gamma-1}\leq\frac{C}{dist(X,\Gamma)^{2}}.
    \end{equation*}
    Consequently, there exists a constant $C_{2}=C_{2}(n,L,\gamma,\delta,\lambda,\Lambda)$, such that
    \begin{equation*}
        \Delta w=c(X)\cdot w\quad\mbox{in }\mathcal{GC}_{2R}\cap\{x_{n}-g(x')\geq\frac{\delta}{2}R\},\quad\mbox{with }0\leq c(X)\leq\frac{C_{2}}{R}.
    \end{equation*}
    Applying the interior Harnack principle in $\mathcal{GC}_{2R}\cap\{x_{n}\geq\frac{\delta}{2}R\}$ gives the estimate \eqref{eq. linearized Harnack 1} with $C_{1}=C_{1}(n,L,\gamma,\delta,\lambda,\Lambda)$.
    
    Similarly, let $\widetilde{w}=\Psi_{k'}-\Psi_{k}\geq0$ for $k'$ given in \eqref{eq. sigma choice to be verified in step 3}. Using exactly the same argument as above, one has
    \begin{equation}\label{eq. linearized Harnack 2}
        C_{1}^{-1}\cdot\Big(\Psi_{k'}(\vec{P})-\Psi_{k}(\vec{P})\Big)\leq\Psi_{k'}(X)-\Psi_{k}(X)\leq C_{1}\cdot\Big(\Psi_{k'}(\vec{P})-\Psi_{k}(\vec{P})\Big)\mbox{ in }\mathcal{SC}_{R}.
    \end{equation}
        
    Now, we claim that there exists a sufficiently large $\underline{R}=\underline{R}(K-k)$, such that for all $R\geq\underline{R}$, one has
    \begin{equation}\label{eq. oscilation reduction, requirement 1}
        \Psi_{k'}(\vec{P})\leq\Psi_{k}(\vec{P})+\frac{\Psi_{K}(\vec{P})-\Psi_{k}(\vec{P})}{2(C_{1})^{2}(\frac{M}{\eta}+1)}
    \end{equation}
    and
    \begin{equation}\label{eq. oscilation reduction, requirement 2}
        \min_{X\in\mathcal{SC}_{R}}\Big(\Psi_{k'}(X)-\Psi_{k}(X)\Big)\geq\eta\max_{X\in\mathcal{GC}_{R}\setminus\mathcal{SC}_{R}}\Big(\Psi_{k'}(X)-\Psi_{k}(X)\Big).
    \end{equation}
    The reason for \eqref{eq. oscilation reduction, requirement 1} and \eqref{eq. oscilation reduction, requirement 2}, as well as the existence of $\underline{R}(K-k)$ will be given later.
    
    Let us temporarily assume the correctness of the claim. Set
    \begin{equation*}
        v(X)=u(X)-\Psi_{k'}(X).
    \end{equation*}
    As $\displaystyle u(\vec{P})\geq\frac{\Psi_{k}(\vec{P})+\Psi_{K}(\vec{P})}{2}$, it follows from \eqref{eq. oscilation reduction, requirement 1} that
    \begin{equation*}
        u(\vec{P})-\Psi_{k}(\vec{P})\geq(C_{1})^{2}(\frac{M}{\eta}+1)\Big(\Psi_{k'}(\vec{P})-\Psi_{k}(\vec{P})\Big).
    \end{equation*}
    By \eqref{eq. linearized Harnack 1} and \eqref{eq. linearized Harnack 2}, the estimate above implies
    \begin{equation*}
        u(X)-\Psi_{k}(X)\geq(\frac{M}{\eta}+1)\Big(\Psi_{k'}(X)-\Psi_{k}(X)\Big)\quad\mbox{in }\mathcal{SC}_{R}.
    \end{equation*}
    In other words,
    \begin{equation}\label{eq. v(X) lower bound from step 1}
        v(X)\geq\frac{M}{\eta}\Big(\Psi_{k'}(X)-\Psi_{k}(X)\Big)\quad\mbox{in }\mathcal{SC}_{R}.
    \end{equation}
    On the other hand, we notice that the assumption $u(X)\geq\Psi_{k}(X)$ implies
    \begin{equation*}
        \Psi_{k'}(X)-\Psi_{k}(X)\geq-v(X)\quad\mbox{in }\mathcal{GC}_{R}.
    \end{equation*}
    Combining this estimate with \eqref{eq. oscilation reduction, requirement 2} and \eqref{eq. v(X) lower bound from step 1} yields that
    \begin{align*}
        \min_{X\in\mathcal{SC}_{R}}v(X)\geq&\frac{M}{\eta}\min_{X\in\mathcal{SC}_{R}}\Big(\Psi_{k'}(X)-\Psi_{k}(X)\Big)\\
        \geq&M\max_{X\in\mathcal{GC}_{R}\setminus\mathcal{SC}_{R}}\Big(\Psi_{k'}(X)-\Psi_{k}(X)\Big)\geq M\cdot\max_{X\in\mathcal{GC}_{R}}\Big(-v(X)\Big).
    \end{align*}
    If $\displaystyle\max_{X\in\mathcal{GC}_{R}}\Big(-v(X)\Big)\leq0$, then \eqref{eq. lower bound update to k'} is already verified. Suppose otherwise that $A:=\max_{X\in\mathcal{GC}_{R}}\Big(-v(X)\Big)>0$, then we let $\hat{v}=v/A$. Clearly,
    \begin{equation*}
        \hat{v}\geq M\mbox{ in }\mathcal{SC}_{R},\quad\mbox{and }\hat{v}\geq-1\mbox{ in }\mathcal{GC}_{R}.
    \end{equation*}
    By setting $u_{1}=u$ and $u_{2}=\Psi_{k'}$ in \eqref{eq. lagrange MVT}-\eqref{eq. lagrange MVT coefficient -2}, we conclude that
    \begin{equation*}
        \Delta\hat{v}=c(X)\hat{v}\mbox{ in }\mathcal{GC}_{R},\quad\mbox{where }0\leq c(X)\leq\frac{C_{1}}{dist(X,\Gamma)^{2}}.
    \end{equation*}
    Applying Lemma~\ref{lem. ensure positive} implies that $v(X)\geq0$ in $\mathcal{GC}_{R/2}$. Hence \eqref{eq. lower bound update to k'} is verified.
        
    It remains to show \eqref{eq. oscilation reduction, requirement 1} and \eqref{eq. oscilation reduction, requirement 2}, as well as the existence of $\underline{R}(K-k)$. For simplicity, we denote
    \begin{equation*}
        w_{1}(X)=\Psi_{k'}(X)-\Psi_{k}(X),\quad w_{2}(X)=\Psi_{K}(X)-\Psi_{k}(X).
    \end{equation*}
    Then \eqref{eq. oscilation reduction, requirement 1} and \eqref{eq. oscilation reduction, requirement 2} can be interpreted as
    \begin{equation}\label{eq. requirements (interpreted)}
        w_{1}(\vec{P})\leq2\sigma w_{2}(\vec{P})\quad\mbox{and }\min_{X\in\mathcal{SC}_{R}}w_{1}(X)\geq\eta\max_{X\in\mathcal{GC}_{R}\setminus\mathcal{SC}_{R}}w_{1}(X).
    \end{equation}
        
    To see this, we can combine the estimates \eqref{eq. Psi_K best estimate}, \eqref{eq. Psi0<=Phi}, and \eqref{eq. sigma choice to be verified in step 3}, and get that
    \begin{align*}
        \sigma(K-k)H(X)-\Phi(X)\leq&w_{1}(X)\leq\sigma(K-k)H(X)+\Phi(X),\\
        (K-k)H(X)-\Phi(X)\leq&w_{2}(X)\leq(K-k)H(X)+\Phi(X).
    \end{align*}
        
    Let $\vec{Q}=\vec{e}_{n}=\frac{1}{R}\cdot\vec{P}$. Using the assumption $\alpha<\phi$, we see that as long as
    \begin{equation}\label{R(K-k) require 1}
        R\geq\Big[\frac{10}{\sigma(K-k)}\cdot\frac{\Phi(\vec{Q})}{H(\vec{Q})}\Big]^{\frac{1}{\phi-\alpha}},
    \end{equation}
    the homogeneity of $H(X)$ and $\Phi(X)$ implies that
    \begin{equation*}
        \frac{w_{1}(\vec{P})}{w_{2}(\vec{P})}\leq\frac{\sigma(K-k)H(\vec{P})+\Phi(\vec{P})}{(K-k)H(\vec{P})-\Phi(\vec{P})}=\frac{\sigma(K-k)H(\vec{Q})+\Phi(\vec{Q})R^{\alpha-\phi}}{(K-k)H(\vec{Q})-\Phi(\vec{Q})R^{\alpha-\phi}}\leq2\sigma.
    \end{equation*}
    This verifies the first inequality of \eqref{eq. requirements (interpreted)}.
    
    Next, for any $X\in\mathcal{SC}_{R}$ and $Y\in\mathcal{GC}_{R}\setminus\mathcal{SC}_{R}$, we let
    \begin{equation*}
        X_{0}=\frac{X}{R}\in\mathcal{SC}_{1,\delta},\quad Y_{0}=\frac{Y}{R}\in\mathcal{GC}_{1}\setminus\mathcal{SC}_{1,\delta}.
    \end{equation*}
    We then have
    \begin{equation*}
        \frac{w_{1}(X)}{w_{1}(Y)}\geq\frac{\sigma(K-k)H(X_{0})-\Phi(X_{0})R^{\alpha-\phi}}{\sigma(K-k)H(Y_{0})+\Phi(Y_{0})R^{\alpha-\phi}}.
    \end{equation*}
    Notice that $H$ is strictly positive in $\overline{\mathcal{SC}_{1,\delta}}$, then as long as
    \begin{equation}\label{R(K-k) require 2}
        R\geq\Big[\frac{10}{\sigma(K-k)}\cdot\frac{\sup _{\mathcal{SC}_{1,\delta}}\Phi}{\inf _{\mathcal{SC}_{1,\delta}}H}\Big]^{\frac{1}{\phi-\alpha}},
    \end{equation}
    it holds that
    \begin{equation*}
        \sigma(K-k)H(X_{0})-\Phi(X_{0})R^{\alpha-\phi}\geq\frac{1}{2}\Big(\sigma(K-k)H(X_{0})+\Phi(X_{0})R^{\alpha-\phi}\Big).
    \end{equation*}
    By recalling the property of $\eta$ in \eqref{eq. choice of eta}, we have verified the second inequality of \eqref{eq. requirements (interpreted)}.

    Finally, the existence of $\underline{R}(K-k)$ follows from \eqref{R(K-k) require 1} and \eqref{R(K-k) require 2}. Hence the claim is verified, and we have finished the proof of Lemma~\ref{lem. nonlinear BHP, if case 1 holds}.
\end{proof}

Similarly, we have the following upper bound improvement. The proof is symmetric to that of Lemma~\ref{lem. nonlinear BHP, if case 1 holds}, and the details are omitted.
\begin{lemma}\label{lem. nonlinear BHP, if case 2 holds}
    Let $(\delta,M,\eta)$ be the constants mentioned in Lemma~\ref{lem. ensure positive} and \eqref{eq. choice of eta}. Assume that $\alpha<\phi$, $0<\lambda\leq f(X)\leq\Lambda$, and $f(X)$ is locally Dini continuous in $Cone_{\Sigma}$. There exist a decreasing function $\underline{R}:(0,+\infty)\to(0,+\infty)$, such that the following holds: For any $0\leq k<K$, and any $R\geq\underline{R}(K-k)$, if $u$ satisfies
    \begin{equation*}
        -\Delta u=f(X)\cdot u^{-\gamma}\mbox{ and }u(X)\leq\Psi_{K}(X)\quad\mbox{in }\mathcal{GC}_{2R},\quad u=0\mbox{ on }\Gamma=\partial Cone_{\Sigma},
    \end{equation*}
    and also satisfies
    \begin{equation*}
        u(\vec{P})\leq\frac{\Psi_{k}(\vec{P})+\Psi_{K}(\vec{P})}{2}\quad\mbox{for }\vec{P}=R\vec{e}_{n}.
    \end{equation*}
    Then, there exists a universal constant $C_{1}=C_{1}(n,L,\gamma,\delta,\lambda,\Lambda)$, such that
    \begin{equation*}
        C_{1}^{-1}\cdot\Big(\Psi_{K}(\vec{P})-u(\vec{P})\Big)\leq\Psi_{K}(X)-u(X)\leq C_{1}\cdot\Big(\Psi_{K}(\vec{P})-u(\vec{P})\Big)\mbox{ in }\mathcal{SC}_{R}.
    \end{equation*}
    Furthermore, denote
    \begin{equation*}
        K'=K-\sigma(K-k),\quad\mbox{where }\sigma=\frac{1}{4(C_{1})^{2}(\frac{M}{\eta}+1)},
    \end{equation*}
    then one has
    \begin{equation}\label{eq. upper bound update to K'}
        u(X)\leq\Psi_{K'}(X)\mbox{ in }\mathcal{GC}_{R/2}.
    \end{equation}
\end{lemma}

Finally, we finish the proof of Lemma~\ref{lem. reduction of oscillation}.
\begin{proof}[Proof of Lemma~\ref{lem. reduction of oscillation}]
    Let $\vec{P}=R\vec{e}_{n}$. At $\vec{P}$ we have
    \begin{equation}\label{eq. alternative}
        \mbox{either }u(\vec{P})\geq\frac{\Psi_{k}(\vec{P})+\Psi_{K}(\vec{P})}{2},\quad\mbox{or }u(\vec{P})\leq\frac{\Psi_{k}(\vec{P})+\Psi_{K}(\vec{P})}{2}.
    \end{equation}
    If the first case of \eqref{eq. alternative} holds, then it follows from Lemma~\ref{lem. nonlinear BHP, if case 1 holds} that \eqref{eq. lower bound update to k'} holds with $K-k'\leq(1-\sigma)\cdot(K-k)$. Similarly, if the second case of \eqref{eq. alternative} holds, then \eqref{eq. upper bound update to K'} holds with $K'-k\leq(1-\sigma)\cdot(K-k)$. This completes the proof of Lemma~\ref{lem. reduction of oscillation}.
\end{proof}
\noindent \textbf{Acknowledgment.}
Guo is partially supported by the NSFC-12501145,  the Natural Science Foundation of Shanghai (No. 25ZR1402207), the Postdoctoral Fellowship Program of CPSF (No. GZC20252004), and  the China Postdoctoral Science Foundation (No. 2025T180838 and 2025M773061). 
Li is partially supported by NSFC-W2531006, NSFC-12250710674, NSFC-12031012 and the Institute of Modern Analysis-A Frontier Research Center of Shanghai. %Zhang is partially supported by NSFC-12526202, and NSFC-12141105.

\noindent{\bf{Data availability statement:}} Data will be made available on reasonable request.

\noindent{\bf{Conflict of interest statement:}} There is no conflict of interest.

%\noindent{\bf{AI Disclosure Statement:}} 
%The authors declare that no artificial intelligence (AI) tools were used in the preparation of this paper.
\bibliographystyle{abbrv}
\bibliography{main}
\end{document}